\documentclass[10pt]{article}
\usepackage[square,numbers]{natbib}
\usepackage{amsthm}
\usepackage{authblk}
\usepackage{graphicx}
\usepackage{amssymb}
\usepackage{amsthm}
\usepackage{amsmath}
\usepackage{mathtools}
\usepackage{comment}

\usepackage[T1]{fontenc}
\usepackage[utf8]{inputenc}
\usepackage{tabularx,booktabs,caption}

\newtheorem{theorem}{Theorem}[section]

\newtheorem{lemma}[theorem]{Lemma}

\begin{document}

\title{ A P\'{o}lya--Vinogradov inequality for short character sums}
\author{Matteo Bordignon}
\affil[]{School of Science, University of New South Wales Canberra }
\affil{m.bordignon@student.unsw.edu.au}
\date{\vspace{-5ex}}
\maketitle

\begin{abstract}
In this paper we obtain a variation of the P\'{o}lya--Vinogradov inequality with the sum restricted to a certain height. 
Assume $\chi$ to be a primitive character modulo $q$, $ \epsilon >0$ and $N\le q^{1-\gamma}$, with $0\le \gamma \le 1/3$. We prove that
\begin{equation*}
\left|\sum_{n=1}^N \chi(n) \right|\le c(\frac{1}{3}-\gamma+\epsilon)\sqrt{q}\log q
\end{equation*}
with $c=2/\pi^2$ if $\chi$ is even and $c=1/\pi$ if $\chi$ is odd. The result is based on the work of Hildebrand \citep{Hildebrand} and Kerr \cite{Bryce}.

\end{abstract}
\section{Introduction}
It is of high interest studying the possible upper bounds of the following quantity
\begin{equation*}
S(N, \chi):=\left|\sum_{n=1}^N \chi(n) \right|,
\end{equation*}
with $N\in \mathbb{N}$ and $\chi$ a primitive Dirichlet character modulo $q$. The interest is principally on primitive characters as a result for these characters can easily be generalized to all non-principal characters.  The famous P\'{o}lya--Vinogradov inequality tells us that 
\begin{equation*}
S(N, \chi)\ll \sqrt{q}\log q,
\end{equation*}
and aside for the implied constant, this is the best known result. The best known asymptotic constant is $\frac{69}{70\pi 3 \sqrt{3}}+o(1)$, if $\chi$ is even from \cite{Granville} and $\frac{1}{ 3 \pi}+o(1)$ if $\chi$ is odd from \cite{Hild1}. For the best completely explicit constant see \cite{Bordignon}, \cite{B-K} and \cite{F-S}.
It is interesting to note that it appears that $S(N, \chi)$ assumes its maximum for $N\approx q$, see the work by Bober et al. in \cite{Bober} and the one by Hildebrand, Corollary 3 of \cite{Hild1}, which proves that for even characters we have that $N=o(q)$ implies $S(N,\chi) =o(\sqrt{q}\log q)$. Our work will continue in this direction using tools developed by Hildebrand in \cite{Hildebrand} that will lead us to prove the following result.
\begin{theorem}
\label{PVN}
Let $\chi$ be a primitive character modulo $q$, $\epsilon > 0$ and $N\le q^{1-\gamma}$, with $0\le \gamma \le 1/3$. We have
\begin{equation*}
S(N, \chi)\le c(\frac{1}{3}-\gamma+\epsilon)\sqrt{q}\log q
\end{equation*}
with $c=2/\pi^2$ if $\chi$ is even and $c=1/\pi$ if $\chi$ is odd.
\end{theorem}
It is easy to see that Theorem \ref{PVN}, when $N\le q^{1-\gamma}$, improves on the previous bound in \cite{Granville} for $\gamma\ge 0.0072$ if $\chi$ is even, although in this range both results are asymptotically superseded by Corollary 3 of \cite{Hild1}, and improves on the previous bound in \cite{Hild1} for any $\gamma>0$ if $\chi$ is odd. Theorem \ref{PVN} can be made completely explicit, thought with smaller leading constants, following the approach used in~\cite{Bordignon}. We note that Theorem \ref{PVN} is also interesting as it could be viewed as a result between the standard P\'{o}lya--Vinogradov and the Burgess bound.\\
In Section \ref{s:2} we will introduce a lemma, due to Hildebrand, that gives a good upper bound on a certain Gaussian sum and a second lemma related to the sum of trigonometric functions, that is a generalization of a result due to Young and Pomerance. In Section \ref{s:3} we will then prove Theorem \ref{PVN} drawing inspiration from Hildebrand's work \cite{Hildebrand}. Our main improvement will come from dividing a Gauss sum into three parts, instead of into two as done by Hildebrand, and control the first sum using certain bounds on sine and cosine in that~range.

\section{Two important lemmas}
\label{s:2}
We will need the following result, that is Lemma 3  of \cite{Hildebrand}.
\begin{lemma}
\label{LH}
Let $0<\epsilon < 1/2$ a fixed number. Then we have, for all primitive characters modulo $q$, $q^{1/3+\epsilon}\le x \le q$ and real $\alpha$,
\begin{equation*}
\left|\sum_{n\le x}\chi(n)e(\alpha n) \right|\ll_{\epsilon}\frac{x}{\log q}.
\end{equation*}
\end{lemma}
We now prove a lemma that is a variation of two results from \cite{Young} and~\cite{Pomerance}.
\begin{lemma}
\label{LI2}
Uniformly for $0 \le \gamma\le 1/3$ and real $\alpha$ we have
\begin{equation}
\label{eq:1}
\sum_{q^{\gamma}\le n \le q^{1/3+\epsilon}}\frac{1-\cos (\alpha n)}{n} \le (\frac{1}{3}-\gamma+\epsilon)\log q +O(1)
\end{equation}
and
\begin{equation}
\label{eq:2}
\sum_{q^{\gamma}\le n \le q^{1/3+\epsilon}}\frac{|\sin (\alpha n)|}{n} \le \frac{2}{\pi} (\frac{1}{3}-\gamma+\epsilon)\log q +O(1).
\end{equation}
\end{lemma}
\begin{proof}
It is easy to show that we can assume $q^{1/3+\epsilon}\ge 5q^{\gamma}+6$. Define
\begin{equation*}
\sigma(a,b)(\alpha):=\sum_{a\le n \le b}\frac{\cos (\alpha n)}{n}.
\end{equation*}
We start proving \eqref{eq:1} noting that, assuming $C\ge 0$, the result would follow from
\begin{equation}
\label{eq:3}
\sigma(q^{\gamma},q^{1/3+\epsilon})(\alpha)\ge -C.
\end{equation}
Taking $v=p(2m+1)$, with $p$ an odd integer and $m$ any integer, equation (10) from \cite{Young} states that
\begin{equation*}
\sigma(m+1,v)(\alpha)\ge \sigma(m+1,v)(\pi)-4-\frac{1}{2q}
\end{equation*}
and noting that
\begin{equation*}
\sigma(m+1,v)(\pi)\ge -1,
\end{equation*}
we have
\begin{equation}
\label{eq:4}
\sigma(m+1,v)(\alpha)\ge -5-\frac{1}{2q}.
\end{equation}
Taking $\overline{m}=\lceil q^{\gamma}\rceil$,
\begin{equation*}
 \overline{p}= 
\begin{dcases}
     \lfloor \frac{q^{1/3+\epsilon}}{2\overline{m}+1}\rfloor - 1 ~~\text{if}~~\lfloor \frac{q^{1/3+\epsilon}}{2\overline{m}+1}\rfloor~~\text{is even} ,\\
    \lfloor \frac{q^{1/3+\epsilon}}{2\overline{m}+1}\rfloor ~~\text{if}~~\lfloor \frac{q^{1/3+\epsilon}}{2\overline{m}+1}\rfloor~~\text{is odd},
\end{dcases}
\end{equation*}
 and $\overline{v}=\overline{p}(2\overline{m}+1)$ , we can rewrite the sum in \eqref{eq:3}, as we assumed $q^{1/3+\epsilon}\ge 5q^{\gamma}+6$, as follows
\begin{equation}
\label{eq:111}
\sigma(q^{\gamma},q^{1/3+\epsilon})(\alpha)=\sigma(q^{\gamma},\overline{v})(\alpha)+\sigma(\overline{v},q^{1/3+\epsilon})(\alpha)-\frac{\cos (\alpha \overline{v})}{\overline{v}}.
\end{equation}
Observing that $q^{1/3+\epsilon}-\overline{v}\le 4\overline{m}+2$ and using $q^{1/3+\epsilon}\ge 5q^{\gamma}+6$, we obtain
\begin{equation}
\label{eq:222}
\left|\sum_{\overline{v}\le n \le q^{1/3+\epsilon}}\frac{\cos (\alpha n)}{n}\right| \leq \frac{q^{1/3+\epsilon}-\overline{v}}{\overline{v}} \le 2.
\end{equation}
Thus \eqref{eq:3} follows applying \eqref{eq:4} and \eqref{eq:222} to \eqref{eq:111}.\newline
We are thus left with proving \eqref{eq:2}. We need the Fourier expansion for $|\sin(\theta)|$ that can be found in \cite{Polya} Problem 34 in Part VI,
\begin{equation*}
|\sin(\theta)|=\frac{2}{\pi}-\frac{4}{\pi}\sum_{m=1}^{\infty}\frac{\cos 2m\theta}{4m^2-1}.
\end{equation*}
Thus,
\begin{equation*}
\sum_{q^{\gamma}\le n \le q^{1/3+\epsilon}}\frac{|\sin (\alpha n)|}{n} =\frac{2}{\pi}\sum_{q^{\gamma}\le n \le q^{1/3+\epsilon}}\frac{1}{n}-\frac{4}{\pi}\sum_{m=1}^{\infty}\frac{1}{4m^2-1}\sum_{q^{\gamma}\le n \le q^{1/3+\epsilon}}\frac{\cos 2mn\alpha}{n},
\end{equation*}
and \eqref{eq:2} now follows from \eqref{eq:3}.
\end{proof}
We can note that, in the above result, the $O(1)$ could be easily made explicit, but for the following applications this is not necessary.
\section{Proof of Theorem \ref{PVN}}
\label{s:3}
We take $\chi$ primitive and start  with
\begin{equation*}
\chi(n)=\frac{1}{\tau (\overline{\chi})}\sum_{a=1}^q \overline{\chi}(a) e\left( \frac{an}{q} \right) = \frac{1}{\tau (\overline{\chi})} \sum_{0< |a|< q/2} \overline{\chi}(a) e\left( \frac{an}{q} \right),
\end{equation*}
where $\tau (\overline{\chi})$ is the Gaussian sum. Summing over $ 1 \le n \le N$, we obtain
\begin{equation*}
\sum_{n=1}^N \chi(n)= \frac{1}{\tau (\overline{\chi})} \sum_{0< |a|< q/2} \overline{\chi}(a) \sum_{n=1}^N e\left( \frac{an}{q} \right)= \frac{1}{\tau (\overline{\chi})} \sum_{0< |a|< q/2} \overline{\chi}(a) \frac{ e\left( \frac{aN}{q} \right)-1}{1- e\left( \frac{-a}{q} \right)}.
\end{equation*}
Since $|\tau (\overline{\chi})|=\sqrt{q}$ for primitive characters and 
\begin{equation*}
\frac{ 1}{1- e\left( \frac{-a}{q} \right)} = \frac{q}{2 \pi i a} +O(1),
\end{equation*}
for $0< |a|< q/2$, it follows that
\begin{equation*}
S(N, \chi)\le \frac{\sqrt{q}}{2 \pi} \left| \sum_{0< |a|< q/2}\frac{\overline{\chi(a)}\left( e(\frac{aN}{q})-1\right)}{a}\right| + O( \sqrt{q}).
\end{equation*}
Now we split the inner sum in three parts: $\Sigma_1$ with $0<|a|< q^{\gamma}/(2\pi)$, $\Sigma_2$ with $q^{\gamma}/(2\pi)\le |a|\le q^{1/3+\epsilon}$ and $\Sigma_3$ with $q^{1/3+\epsilon}<|a|< q/ 2$. Therefore
\begin{equation}
\label{eq:fin}
S(N, \chi)\le \frac{\sqrt{q}}{2 \pi} \left( \left| \sum_1\right|+ \left| \sum_2\right| +\left| \sum_3\right|\right) + O( \sqrt{q}).
\end{equation}
We have
\begin{equation*}
 \Sigma_1= 
\begin{dcases}
     2i\sum\limits_{1\le a < q^{\gamma}/(2\pi)}\frac{\overline{\chi(a)}\sin(\frac{2\pi aN}{q})}{a} ~~\text{if}~~\chi~~\text{is even} ,\\
    -2\sum\limits_{1\le a < q^{\gamma}/(2\pi)}\frac{\overline{\chi(a)}\left( 1-\cos(\frac{2\pi aN}{q})\right)}{a}~~\text{if}~~\chi~~\text{is odd}.
\end{dcases}
\end{equation*}
Now observing that for $0<|a|< q^{\gamma}/(2\pi)$ we have that $|\frac{2\pi aN}{q}|<1$ and thus 
\begin{equation*}
\left| \sin\left(\frac{2\pi aN}{q}\right)\right| \ll \frac{2\pi aN}{q} \quad \text{and} \quad \left|1-\cos\left(\frac{2\pi aN}{q}\right) \right|\ll \left(  \frac{2\pi aN}{q}\right)^2,
\end{equation*}
where, remembering that $N\le q^{1-\gamma}$, easily give $\Sigma_1\ll 1$.\\
We also have
\begin{equation*}
 \Sigma_2= 
\begin{dcases}
     2i\sum\limits_{q^{\gamma}/(2\pi)\le a \le q^{1/3+\epsilon}}\frac{\overline{\chi(a)}\sin(\frac{2\pi aN}{q})}{a} ~~\text{if}~~\chi~~\text{is even},\\
    -2\sum\limits_{q^{\gamma}/(2\pi)\le a \le q^{1/3+\epsilon}}\frac{\overline{\chi(a)}\left( 1-\cos(\frac{2\pi aN}{q})\right)}{a}~~\text{if}~~\chi~~\text{is odd},
\end{dcases}
\end{equation*}
and applying Lemma \ref{LI2}, observing that the difference between having the lower bound of the sum starting from $q^{\gamma}/(2\pi)$ and not $q^{\gamma}$ is bounded by a constant, we obtain
\begin{equation*}
 \Sigma_2= 
\begin{dcases}
      \frac{4}{\pi}(\frac{1}{3}-\gamma+\epsilon)\log q +O(1) ~~\text{if}~~\chi~~\text{is even},\\
     2(\frac{1}{3}-\gamma+\epsilon)\log q +O(1)~~\text{if}~~\chi~~\text{is odd}.
\end{dcases}
\end{equation*}
By partial summation and Lemma \ref{LH} we have
\begin{align*}
\left|\Sigma_3 \right| \ll (\log q) \max_{q^{1/3+\epsilon} \le x \le q} \left|\frac{1}{x}\sum_{a\le x} \overline{\chi}(a)\left( e\left(\frac{aN}{q}\right)-1\right) \right|\ll_{\epsilon} 1.
\end{align*}
Thus, by the above upper bounds on $|\sum_i|$, for $i=1,2,3$, and \eqref{eq:fin}, we obtain the desired result.
\section*{Acknowledgements}
I would like to thank my supervisor Tim Trudgian for his kind help and his suggestions in developing this paper.


\begin{thebibliography}{9}

\bibitem {Bober}
J. Bober, L. Goldmakher, A. Granville and D. Koukoulopoulos,
\textit{The frequency and the structure of large character sums}, J. Eur. Math. Soc. (JEMS), 20(7): 1759--1818, 2018.
	
\bibitem{Bordignon}
M. Bordignon,
\textit{Partial Gaussian sums and the P\'{o}lya--Vinogradov inequality for primitive characters}, 2020,
arXiv:2001.05114 .

\bibitem{B-K}
M. Bordignon and B. Kerr,
\textit{An explicit P\'{o}lya--Vinogradov inequality via Partial Gaussian sums}, Trans. Amer. Math. Soc., Accepted, 2020, https://doi.org/10.1090/tran/8138 .

\bibitem{Bryce}
B. Kerr,
\textit{On the constant in the {P}\'{o}lya-{V}inogradov inequality},
J. Number Theory, (212): 265--284, 2020. 

\bibitem {F-S}
D. A. Frolenkov and K. Soundararajan,
\textit{A generalization of the {P}\'olya-{V}inogradov inequality.}
Ramanujan J., 31(3): 271--279, 2013.

\bibitem{Granville}
A. Granville and K. Soundararajan,
\textit{Large character sums: pretentious characters and the
              {P}\'{o}lya-{V}inogradov theorem}, J. Amer. Math. Soc, 20(2): 357--384, 2007.

\bibitem{Hildebrand}
A. Hildebrand,
\textit{On the constant in the {P}\'{o}lya-{V}inogradov inequality}, Canad. Math. Bull., 31(3): 347--352, 1988.

\bibitem{Hild1}
A. Hildebrand,
\textit{Large values of character sums}, J. Number Theory, 29(3): 271--296, 1988.

\bibitem{M-V}
H. L. Montgomery and R. C. Vaughan, 
\textit{Exponential sums with multiplicative coefficients}, Invent. Math., 43(1): 69--82, 1977.


\bibitem{Paley}
R. E. A. C. Paley, {\it A theorem on characters}, J. London Math. Soc., (7): 28--32, 1932. 

\bibitem{Polya}
G. P\'{o}lya and G. Szeg\"{o}, {\it Problems and Theorems in Analysis. {II}}, Springer-Verlag, Berlin, Heidelberg 1998.

\bibitem{Pomerance}
C. Pomerance,
\textit{Remarks on the {P}\'{o}lya-{V}inogradov inequality},
Integers, 11(4): 531--542, 2011. 

\bibitem{Young}
W. H. Young,
\textit{On a {C}ertain {S}eries of {F}ourier},
Proc. London Math. Soc. (2), (11): 357--366, 1913. 



\end{thebibliography}
\end{document}